\documentclass[11pt]{amsart}
\usepackage{amsmath}
\usepackage{amsthm,amssymb}
\usepackage{hyperref}
\usepackage{amssymb}
\usepackage{graphicx}
\usepackage[utf8]{inputenc}
\input xy
\xyoption{all}
\usepackage{color}
\usepackage{enumerate}
\newtheorem{theorem}{Theorem}[section]

\newtheorem{lemma}[theorem]{Lemma}

\newtheorem{corollary}[theorem]{Corollary}
\newtheorem{proposition}[theorem]{Proposition}
\theoremstyle{definition}

\newtheorem{remark}[theorem]{Remark}
\newtheorem{problem}[theorem]{Problem}

\newcommand{\ch}{\mbox{\rm char}}

\newcommand{\Z}{\mathbb{Z}}
\newcommand{\C}{\mathbb{C}}

\begin{document}

\title[On semi-nil clean rings ...]{On semi-nil clean rings with applications}

\author[M. H. Bien]{M. H. Bien$^{1,2}$}
\author[P. V. Danchev]{P. V. Danchev$^{3}$}
\author[M. Ramezan-Nassab]{M. Ramezan-Nassab$^{4,5}$}

\address{[1] Faculty of Mathematics and Computer Science, University of Science, Ho Chi Minh City, Vietnam}

\address{[2] Vietnam National University, Ho Chi Minh City, Vietnam}

\address{[3] Institute of Mathematics and Informatics, Bulgarian Academy of Sciences, 1113 Sofia, Bulgaria}

\address{[4] Department of Mathematics, Kharazmi University, 50 Taleghani Street, Tehran, Iran} ~~
    \address{[5] School of Mathematics, Institute for Research in Fundamental Sciences (IPM), P.O. Box 19395-5746, Tehran, Iran.} ~~~~\\

\medskip
\medskip

\email{M. H. Bien: mhbien@hcmus.edu.vn \newline
        P. V. Danchev: danchev@math.bas.bg; pvdanchev@yahoo.com \newline
        M. Ramezan-Nassab: ramezann@khu.ac.ir}

\keywords{Clean ring; (Weakly) Periodic ring; Semi-nil clean ring, Group ring; Unit group.\\
\protect
\protect
Corresponding author: Peter V. Danchev.}

\begin{abstract}
We investigate the notion of \textit{semi-nil clean} rings, defined as those rings in which each element can be expressed as a sum of a periodic and a nilpotent element. Among our results, we show that if $R$ is a semi-nil clean NI ring, then $R$ is periodic. Additionally, we demonstrate that every group ring $RG$ of a nilpotent group $G$ over a weakly 2-primal ring $R$ is semi-nil clean if, and only if, $R$ is periodic and $G$ is locally finite.

Moreover, we also study those rings in which every unit is a sum of a periodic and a nilpotent element, calling them \textit{unit semi-nil clean} rings. As a remarkable result, we show that if $R$ is an algebraic algebra over a field, then $R$ is unit semi-nil clean if, and only if, $R$ is periodic.

Besides, we explore those rings in which non-zero elements are a sum of a torsion element and a nilpotent element, naming them \textit{t-fine} rings, which constitute a proper subclass of the class of all fine rings. One of the main results is that matrix rings over t-fine rings are again t-fine rings.
\end{abstract}

\maketitle


\section{Introduction and Motivation}

All rings considered in this paper are unitary (i.e., containing an identity element) and associative. Recall that an element $x$ in a ring $R$ is \textit{periodic} if there exist two different natural numbers $m$ and $n$ such that $x^m=x^n$. A periodic ring is the one in which each of its elements is periodic. More generally, a ring $R$ is called {\it weakly periodic} if every $x$ in $R$ can be written in the form $x = a + b$ for some potent element $a$ (i.e., $a^n=a$ for a positive integer $n$ depending on $a$) and some nilpotent element $b$ in $R$. It is well known that every periodic ring is weakly periodic (see \cite[Theorem~10.1.1]{Sheibani}), but according to \cite[Examples 3.1 and 3.2]{Ser}, the converse is manifestly {\it not} true.

Imitating \cite{Bisht}, an element $x$ in a ring $R$ is said to be {\it semi-nil clean} element if it can be written in the form $x = a + b$, where $a$ is a periodic element and $b$ is a nilpotent element. In addition, if $ab=ba$ in $R$, the element $x$ is said to be {\it strongly semi-nil clean}. A (strongly) semi-nil clean ring is the one in which all its elements are (strongly) semi-nil clean. It is obvious that each weakly periodic ring is semi-nil clean, but at the moment, we do not know whether the two classes of these rings differ each other or are the same.

Let $R$ be a semi-nil clean ring. For $x\in R$, we write $x-1=a+b$, where $a$ is periodic and $b$ is nilpotent. Thus, $x$ can be written as the sum of a periodic element $a$ and a unipotent element $b+1$ in $R$. Therefore, each semi-nil clean ring $R$ is a {\it semi-clean} ring, in the sense that each element of $R$ can be written as the sum of a periodic element and a unit element in $R$. But, however, there are semi-clean rings that are {\it not} semi-nil clean. To give an example, let $R$ be a finite ring. Then, the power series ring $R[[t]]$ is semi-clean in accordance with \cite[Proposition~3.3]{Ye}, while it is {\it not} semi-nil clean, because the central element $t$ is not periodic (see Proposition~\ref{center} stated below). As another example, inspired by \cite{Ye}, the group ring $\mathbb{Z}_{(p)}C_3$ is semi-clean, where $p$ is a prime integer, $C_3$ is a cyclic group of order 3, and $\mathbb{Z}_{(p)}$ is the localization of $\mathbb{Z}$ at $p$. However, $\mathbb{Z}_{(p)}C_3$ is definitely {\it not} semi-nil clean utilizing Theorem~\ref{local-nilpotent} quoted below.

\medskip

Thus, we have the following inclusion relationships between these classes of rings:

$$\{\text{periodic} \}\subsetneqq \{\text{weakly periodic}\} \subseteq\{\text{semi-nil clean}\}\subsetneqq\{\text{semi-clean}\}.$$

\medskip
As already noticed, we are unable to decide presently whether the middle inclusion above is strict or not. So, we now arrive at our motivating question.

\begin{problem}\label{1}
Is each semi-nil clean ring weakly periodic?
\end{problem}

Note that the class of all semi-nil clean rings also properly contains the class of all {\it nil clean} rings, i.e., the class of rings $R$ in which each of its elements is a sum of an idempotent and a nilpotent in $R$ (for instance, each finite field is semi-nil clean but {\it not} nil clean).

\medskip

On the other side, it is worthwhile noticing that all semi-nil clean rings are always additively $2$-periodic as defined in \cite{Additively}.

\medskip

Our principal work is organized as follows: in the next Section~\ref{sec2}, the basic properties of semi-nil clean rings are investigated. We show the curious fact that each strongly semi-nil clean ring is periodic (Proposition~\ref{strongly}). Let $R$ be a semi-nil clean ring. As some remarkable results, we show that if $R$ is either an NI ring, or has only finitely many non-central nilpotent elements, then $R$ is periodic (see Theorems~\ref{NI} and ~\ref{FN} as well as Proposition~\ref{newnil}, respectively).

In the subsequent Section~\ref{sec3}, we apply our results to some aspects of semi-nil clean group rings. We will show that every group ring $RG$ of a nilpotent group $G$ over a weakly 2-primal ring $R$ is semi-nil clean if, and only if, $RG$ is periodic if, and only if, $R$ is semi-nil clean and $G$ is locally finite (Theorem~\ref{local-nilpotent}). Likewise, there is a similar result provided $R$ is a local ring whose unit group is locally nilpotent (Theorem~\ref{local}).

In Section~\ref{sec4}, we study those rings $R$ in which every unit is a sum of a periodic element and a nilpotent element, which we call {\it unit semi-nil clean} rings. Among other results, we will prove that if $R$ is an algebraic algebra over a field, then $R$ is unit semi-nil clean if, and only if, $R$ is periodic (Theorem~\ref{algebraic2}). Some ring extensions of unit semi-nil clean rings, such as polynomial rings and matrix rings, are explored as well.

Recall that, as in \cite{CL}, a ring $R$ is said to be \textit{fine} if every non-zero element in it is a sum of a unit and a nilpotent. As a proper subclass of both the class of all fine rings and the class of all weakly periodic rings, we explore in Section~\ref{sec5} those rings in which all non-zero elements are a sum of a torsion element and a nilpotent element, naming them \textit{t-fine} rings. One of the main results is that matrix rings over t-fine rings are too t-fine rings (Theorem~\ref{matrix3}). We also show that a for a commutative ring $R$, the matrix ring ${\rm M}_n(R)$ is a t-fine ring if, and only if, $R$ is a locally finite field (Corollary~\ref{comm-fine}).


\medskip

\section{Main Results}\label{sec2}

If $R$ is a ring, the center, the Jacobson radical and the set of all nilpotent elements of $R$ are denoted by ${\rm Z}(R)$, ${\rm J}(R)$ and ${\rm Nil}(R)$, respectively. Also, ${\rm P}(R)$ and $\mathfrak{u}(R)$ stand for the set of all periodic elements and the set of all unipotent elements in $R$, respectively. We denote the unit group of $R$ by ${\rm U}(R)$.

\medskip

We observe that, if the ring $R$ is semi-nil clean, i.e., $R={\rm P}(R)+{\rm Nil}(R)$, then the equality $R={\rm P}(R)+\mathfrak{u}(R)$ surely also holds.

\medskip

We now begin with the following lemma, which is similar to \cite[Lemma~2.2]{Additively}. As we will need the idea of its proof later on in this paper, we include its proof for the sake of completeness and the readers' convenience.

\begin{lemma}\label{center1}
If $R$ is a ring such that each of its elements is a sum of two elements in ${\rm P}(R)\cup\mathfrak{u}(R)$, then $\ch(R)>0$ and $\mathfrak{u}(R)\subseteq{\rm P}(R)$.
\end{lemma}

\begin{proof}
If $R$ is of zero characteristic, then $\Z$ can be viewed as a subring of $R$. By hypothesis, there are two elements $u$ and $v$ in ${\rm P}(R)\cup\mathfrak{u}(R)$ such that $3=u+v$, and, as $u+v$ is central in $R$, we have $uv=vu$. Thus, the subring $B=\Z[u,v]$ of $R$ generated by $u$ and $v$ over $\Z$ is necessarily a commutative ring which is integral over $\Z$ (see \cite[Proposition~5.1 and Corollary~5.3]{Atiyah}).

Now, let $f:\Z \to \C$ be the inclusion homomorphism. So, invoking \cite[Exercise~2, p.~67]{Atiyah}, there exists a ring homomorphism $\alpha:B\to \C$ such that $\alpha|_\Z=f$. Therefore, $3=\alpha(3)=\alpha(u)+\alpha(v)$. Note that, if $x$ is either a periodic or a unipotent element of $\C$, then the absolute value $|x|$ is either $0$ or $1$. Consequently, $3= |\alpha(u)+\alpha(v)|\leq 2$, which is a contradiction. Hence, the characteristic of $R$ is positive.

Now, suppose that $\ch(R)=m>0$ and $u\in \mathfrak{u}(R)$. Then, owing to \cite[Theorem~3.2]{Lam-Danchev}, we obtain $u^{m^s}=1$ for some integer positive $s$. Thus, $u\in {\rm P}(R)$, as required.
\end{proof}


\begin{lemma}\label{center0}\cite[Lemma~2.6]{Bou}
Let $R$ be a ring of positive characteristic and let $x\in R$. If $x = a + b$ for some $a, b\in{\rm P}(R)$ and $ab=ba$, then $x\in {\rm P}(R)$.
\end{lemma}


The two lemmas above lead to the following surprising result.

\begin{proposition}\label{strongly}
A ring $R$ is strongly semi-nil clean if and only if $R$ is periodic.
\end{proposition}


Furthermore, from Lemmas~\ref{center1} and \ref{center0} combined with a result of Herstein in \cite{Herstein} (see also \cite[Theorem~1.3]{Bell}), we deduce the following assertion.

\begin{proposition}\label{center}
Let $R$ be a semi-nil clean ring. The following two items hold:
\begin{itemize}
\item [(1)] Each central subring of $R$ is a periodic ring.
\item[(2)] If ${\rm Nil}(R)\subseteq {\rm Z}(R)$, then $R$ is a commutative periodic ring.
\end{itemize}
\end{proposition}


Recall that a ring $R$ is {\it weakly $2$-primal} if ${\rm Nil}(R)$ is equal to the Levitzki radical of $R$. For example, every commutative or reduced ring is weakly 2-primal. Recall also that an {\it NI ring} is a ring $R$ such that ${\rm Nil}(R)$ forms an ideal of $R$.

\medskip

The next result is one of our key rules in the paper.

\begin{theorem}\label{NI}
If $R$ is a semi-nil clean ring and ${\rm Nil}(R)$ is an additive subgroup of $R$, then ${\rm J}(R)$ is nil. Particularly, $R$ is a semi-nil clean NI ring if and only if ${\rm J}(R)={\rm Nil}(R)$ and $R$ is a periodic ring.
\end{theorem}

\begin{proof}
Let $x\in {\rm J}(R)$. Then, we can find a periodic element $a\in R$ such that $b:=x-a\in {\rm Nil}(R)$. Assume $b^n=0$ for some positive integer $n$. Thus, $a^n\in {\rm J}(R)$. On the other hand, according to \cite[Lemma~5.2]{Ye}, there exists a positive integer $k$ such that $a^k$ is an idempotent. Now, one can write that
$$a^k=\left(a^k\right)^{2^n}=\left(a^k\right)^n \left(a^k\right)^{2^n-n}\in {\rm J}(R).$$
It, therefore, follows that $a^k=0$. Hence, $a$, and so $x=a+b$, lies in ${\rm Nil}(R)$. This proves that ${\rm J}(R)$ is nil. Particularly, if $R$ is an NI ring, then ${\rm J}(R)={\rm Nil}(R)$, as expected.

In order to establish that $R$ is periodic, let $y$ be any element in $R$. Write $y=a+b$, where $a\in {\rm P}(R)$ and $b\in {\rm Nil}(R)$. As above, there exists a natural number $k$ such that $a^k$ is an idempotent. Now, one may write that
$$y^k-y^{2k}=(a+b)^k-(a+b)^{2k}\in {\rm Nil}(R).$$
Consequently, thanks to \cite[Lemma~3.5]{YKZ}, there exists an idempotent $e$ and a nilpotent $c$, both in $R$, such that $y^k=e+c$ and $ec=ce$ (i.e., $y^k$ is a strongly nil clean element). That is why, the result follows at once from \cite[Theorem~3.4]{Cui}, completing the arguments.
\end{proof}


As a  consequence of Theorem~\ref{NI}, we have the following generalization of \cite[Theorem~10.4.5]{Sheibani} or \cite[Corollary]{Bell2}.

\begin{corollary}
Let $R$ be a semi-nil clean ring. If, for each $x,y\in {\rm Nil}(R)$, we have $xy=yx$ (that is, ${\rm Nil}(R)$ is commutative), then $R$ is periodic.
\end{corollary}

\begin{proof}
Exploiting Theorem~\ref{NI}, it suffices to show that ${\rm N:= Nil}(R)$ is an ideal. The following two claims can be extracted from the proof of either \cite[Theorem~10.4.5]{Sheibani} or \cite[Theorem~2]{Bell2}.

\begin{itemize}
\item [(i)] {\it Let $R$ be a ring such that ${\rm N}$ is commutative. If $e\in R$ is an idempotent, then $e{\rm N}\subseteq {\rm N}$.}
\item [(ii)] {\it Let $R$ be a ring such that ${\rm N}$ is multiplicatively closed. If there exist some $x\in R$ and some natural number $k$ such that $x^k{\rm N}\subseteq{\rm N}$, then $x{\rm N}\subseteq {\rm N}$.}
\end{itemize}

Now, let $r\in R$ and write $r=a+b$, where $a\in {\rm P}(R)$ and $b\in {\rm N}$. We know with \cite[Lemma~5.2]{Ye} at hand that $a^k$ is idempotent for some natural number $k$. Thus, $a^k{\rm N}\subseteq {\rm N}$ by~{(i)}, and so $a{\rm N}\subseteq {\rm N}$ by~{(ii)}. Therefore, $r{\rm N}\subseteq {\rm N}$. Similarly, ${\rm N}r\subseteq {\rm N}$, showing that ${\rm N}$ is an ideal (two-sided), as asked for.
\end{proof}


If $R$ is a local ring, then ${\rm Nil}(R)\subseteq J(R)$. If $R$ is a weakly periodic ring, then $J(R)$ is a nil-ideal \cite[Lemma~10.4.1]{Sheibani}. Thus, as an another consequence of Theorem~\ref{NI}, we have the next result.

\begin{corollary}\label{}
Every weakly periodic local  ring is periodic.
\end{corollary}


We now need to know an easy but important fact.

\begin{lemma}\label{lift}
Suppose that $R$ and $S$ are two rings, and let $I$ be a nil-ideal of $R$. Then, the following two points are true:
\begin{itemize}
\item [(1)] $R$ is semi-nil clean (resp., periodic) if and only if $R/I$ is semi-nil clean (resp., periodic).
\item [(2)] $R$ and $S$ are semi-nil clean (resp., periodic) if and only if $R\times S$ is semi-nil clean (resp., periodic).
\end{itemize}
\end{lemma}

\begin{proof}
The first statement follows directly from Lemma~\ref{center1} and \cite[Theorem~5]{lift}, and the second statement follows easily from a simple observation (see the proof of \cite[Proposition~2.12]{Additively}).
\end{proof}


Now we have the following generalization of \cite[Theorem~4.7]{Bell}.

\begin{theorem}\label{FN}
If $R$ is a semi-nil clean ring with only finitely many non-central nilpotent elements, then $R$ is periodic.
\end{theorem}

\begin{proof}
Taking into account Proposition~\ref{center}, we may assume that ${\rm Nil}(R)\nsubseteq{\rm Z}(R)$. If $x\in {\rm Nil}(R)\setminus {\rm Z}(R)$ and $y\in {\rm Nil}(R)\cap {\rm Z}(R)$, then $x+y\in {\rm Nil}(R)\setminus {\rm Z}(R)$. Thus, ${\rm Nil}(R)\cap {\rm Z}(R)$, and hence ${\rm Nil}(R)$ is finite.

Let $\mathcal{P}$ be the prime radical (or, the lower nil-radical) of $R$. Then, knowing \cite[Corollary~5]{Klein}, one writes that $R/\mathcal{P}\cong B\bigoplus C$, where $B$ is a reduced semi-nil clean ring (whence it is a periodic ring) and $C$ is a finite direct sum of full matrix rings over finite fields. But, since $\mathcal{P}$ is nil and $R/\mathcal{P}$ is periodic, Lemma~\ref{lift} forces that $R$ is periodic, as pursued.
\end{proof}


From the above result, it follows immediately that if $R$ is a semi-nil clean ring with only finitely many non-central zero-divisors, then $R$ is periodic (see \cite[Corollary~10.4.11]{Sheibani}).

\medskip

Recall that a {\it PI ring} is a ring that satisfies a polynomial identity. Now, combining Theorem~\ref{NI}, \cite[Theorems~2.4 and 2.10]{ABD}, \cite[Theorem~2.21]{Bou} and \cite[Theorem~1]{Hirano}, we obtain the following result for matrix rings.

\begin{corollary}\label{matrix1}
Let $R$ be a weakly $2$-primal ring or a PI ring, and let $n$ be a natural number. Then, the following statements are equivalent:
\begin{itemize}
\item[(1)] $R$ is semi-nil clean.
\item[(2)] $R$ is periodic.
\item[(3)] ${\rm M}_n(R)$ is periodic.
\end{itemize}

\end{corollary}


Another more general version of Theorem~\ref{NI} is as follows.

\begin{proposition}\label{newnil}
Let  $R$ be a semi-nil clean ring. If ${\rm Nil}(R)$ is an additive subgroup of $R$ having bounded index of nilpotence (that is, there exists a positive integer $d$ such that $x^d=0$ for every $x\in {\rm Nil}(R)$), then $R$ is periodic.
\end{proposition}

\begin{proof}
Thanks to Corollary~\ref{matrix1}, it suffices to show that $R$ is a PI ring. But Theorem~\ref{NI} tells us that ${\rm J}(R)$ is nil. Thus, ${\rm J}(R)$ satisfies $x^d=0$ for some $d \in \mathbb{N}$. Furthermore, if we show that $R/{\rm J}(R)$ is commutative, then $R$ will satisfy the identity $(xy - yx)^d = 0$, and so we will be finished.

Next, to achieve that, we without loss of generality may assume that $R$ is semiprimitive. We then claim that $R$ is commutative. To see this, note that $R$ is a subdirect product of primitive rings $\{R_\alpha\}$, where each $R_\alpha$, being a homomorphic image of $R$, is again a semi-nil clean ring and ${\rm Nil}(R_\alpha)$ remains additively closed. If all of the $R_\alpha$'s are commutative, then so is $R$. Thus, we reduce the assumption to the case where $R$ is itself primitive.

Now, by virtue of the prominent Jacobson Density Theorem (see, e.g., \cite{Lam}), there exists a division ring $D$ such that either $R \cong {\rm M}_n(D)$ for some natural number $n$, or there exists a subring $S$ of $R$ such that ${\rm M}_2(D)$ is a homomorphic image of $S$. Since the nilpotent elements of ${\rm M}_n(D)$ are known to be not additively closed for all $n \geq 2$, it follows at once that that $R \cong D$. Therefore,  $R$ is a commutative field, as claimed, thus proving the whole result.
\end{proof}


Let $R, S$ be two rings, and let $M$ be an $(R,S)$-bi-module. We designate the triangular matrix ring
\begin{align*}
\begin{pmatrix}
    R & M \\
    0 & S
\end{pmatrix}= \left\{ \begin{pmatrix}
    r & m\\
    0 & s
\end{pmatrix}\mid r\in R, s\in S, m\in M\right\}
\end{align*}
by ${\rm T}(R, S,M)$. Also, ${\rm T}_n(R)$ denotes the ring of all upper triangular $n\times n$ matrices over $R$. Then, ${\rm T}(R, S,M)$ (resp., ${\rm T}_n(R)$) is periodic if, and only if, both $R$ and $S$ are periodic (see, for more information, \cite[Theorem~2.13]{Bou} or \cite[Corollary~2.15]{ABD}).

\medskip

Now, combining this observation with Theorem~\ref{NI} and \cite[Corollary~2.13]{Bisht}, we readily obtain the following consequence.


\begin{corollary}
Let $R$ and $S$ be  NI rings, $n\geq 2$, and let $M$ be an $(R,S)$-bi-module. Then, ${\rm T}(R, S,M)$ (resp., ${\rm T}_n(R)$) is semi-nil clean if and only if ${\rm T}(R, S,M)$ (resp., ${\rm T}_n(R)$) is periodic, if and only if $R$ and $S$ are both periodic.
\end{corollary}


We now show the validity of the next claim, which provides a necessary and sufficient condition for when a ring is semi-nil clean or periodic.

\begin{proposition}\label{indecom}
A ring $R$ is semi-nil clean (resp., periodic) if and only if every indecomposable homomorphic image of $R$ is semi-nil clean (resp., periodic).
\end{proposition}

\begin{proof}
The  necessity is quite clear, so we omit the details.

For the sufficiency, suppose that every indecomposable homomorphic image of $R$ is semi-nil clean, and assume on the contrary that $a\in R$ is not semi-nil clean (the proof for ``periodic" is similar, so we drop off the arguments). Then, the set
$$\Sigma=\left\{I\lhd R \mid a+I\in R/I ~\text {is not semi-nil clean}\right\}$$
is not empty. For a chain $\{I_\lambda\}$ of elements of $\Sigma$, let $I=\bigcup_\lambda I_\lambda$. Thus, $I$ is obviously an ideal (two-sided) of $R$.

If now $a+I$ is semi-nil clean in $R/I$, then there exist $b,c\in R$ and $m,n,k\in \mathbb{N}$ such that
\begin{equation}\label{eq1}
a+I=(b+I)+(c+I),~ b^n-b^m\in I~{\rm and}~ c^k\in I.
\end{equation}
But, because  $\{I_\lambda\}$ is a chain, there exists some $\lambda_0$ such that (\ref{eq1}) holds in $R/I_{\lambda_0}$, i.e., $a+I_{\lambda_0}$ is a semi-nil clean element in $R/I_{\lambda_0}$. This contradiction shows that $I\in\Sigma$. Therefore, with the help of the classical Zorn's Lemma, $\Sigma$ has a maximal element, say $L$. It now suffices to show that the factor-ring $R/L$ is indecomposable.

Assume the opposite, namely that $R/L$ is decomposable. So, there exist ideals $K_j\supsetneq L$ of $R$ ($j=1, 2$) such that
$$R/L\cong (R/K_1)\bigoplus (R/K_2), ~{\rm via~the~map}~r+L\mapsto(r+K_1,r+K_2).$$
By the maximality of $L$ in $\Sigma$, the element $a+K_j$ is semi-nil clean in $R/K_j$ for $j=1,2$. Hence, Lemma~\ref{lift} tells us that $a+L$ is semi-nil clean in $R/L$, a contradiction. Consequently, the quotient $R/L$ is really indecomposable, as claimed.
\end{proof}


In closing this section, an important and extremely difficult query is of whether or not each semi-nil clean (in particular, each weakly periodic ring) is clean? If {\it not}, does it follow that each ring which is simultaneously clean and semi-nil clean (in particular, clean and weakly periodic) is also periodic?

\medskip


\section{Semi-Nil Clean Group Rings}\label{sec3}

We shall apply now the results obtained so far to group rings. To this purpose, let $FG$ be a semi-nil clean  group algebra of a group $G$ over a field $F$. If $G$ is torsion and the unit group ${\rm U}(FG)$ satisfies a group identity, then according to the well-known positive answer to Hartley's conjecture, $FG$ satisfies a polynomial identity \cite[Theorem~1.2.27]{Lee}. Therefore, $FG$ is periodic by Corollary~\ref{matrix1}. Specifically, we formulate the following.

\begin{proposition}
Let $G$ be a torsion group and let $F$ be a field such that ${\rm U}(FG)$ satisfies a group identity. If $FG$ is  semi-nil clean, then  $FG$ is periodic.
\end{proposition}


It is worthwhile to indicate also that, if $R$ is a ring and $G$ is a nilpotent group, then $RG$ is nil clean if, and only if, $R$ is a nil clean ring and $G$ is a 2-group (see \cite[Theorem~2.7]{nil-clean GR}).

\medskip

We now deduce the following chief result.

\begin{theorem}\label{local-nilpotent}
Let $R$ be a weakly $2$-primal  ring and let $G$ be a nilpotent group. Then, the following statements are equivalent:
\begin{itemize}
\item [(1)] $RG$ is semi-nil clean.
\item [(2)] $R$ is periodic and $G$ is locally finite.
\item [(3)] $RG$ is periodic.
\end{itemize}
\end{theorem}

\begin{proof}
Taking into account \cite[Theorems~1.2 and 1.6]{ABD}, it suffices to prove only the implication (1)~$\Rightarrow$~(2). To this end, suppose that \( RG \) is a semi-nil clean ring. We know that \( RG/\Delta(G) \cong R \), where \( \Delta(G) \) is the augmentation ideal of \( RG \). Therefore, \( R \) is also a semi-nil clean ring. Utilizing Theorem~\ref{NI}, we conclude that \( R \) is periodic.

Let \( G \) be a nilpotent group of class \( n \). We claim that \( G \) is locally finite. We shall establish our claim by induction on \( n \). If \( n = 1 \), then \( G = \mathrm{Z}(G) \)  is a torsion (= locally finite) group in view of Lemmas~\ref{center1} and \ref{center0}.  Next, assume that the claim is true for all nilpotent groups of class less than \( n \geq 2 \). Note that \( G/\mathrm{Z}(G) \) is a nilpotent group of class \( n - 1 \). In addition, \( R(G/\mathrm{Z}(G)) \), as being a homomorphic image of \( RG \), is still a semi-nil clean ring. Therefore, by the induction hypothesis, the factor-group \( G/\mathrm{Z}(G) \) is locally finite. Since \( \mathrm{Z}(G) \) is also locally finite, we conclude that \( G \) is locally finite, as wanted.
\end{proof}


Now, as a generalization of \cite[Corollary~1.4]{ABD}, we have the following.

\begin{corollary}\label{Cor10}
Let $R$ be a weakly $2$-primal ring and let $G$ be a locally nilpotent group. Then, $RG$ is periodic if and only if $R$ is periodic and $G$ is locally finite.
\end{corollary}


\medskip

Let $R$ be a ring. The set of all torsion elements of ${\rm U}(R)$ is designed by $\mathcal{T}(R)$. Obviously, the containment $\mathcal{T}(R) \subseteq {\rm P}(R)$ holds.


\medskip

We are now ready to offer the following.

\begin{lemma}\label{AC}
If $R$ is a ring of positive characteristic and ${\rm U}(R)$ is a locally nilpotent group, then $\mathcal{T}(R)+\mathcal{T}(R)\subseteq {\rm P}(R)$.
\end{lemma}

\begin{proof}
Recall that, since ${\rm U}(R)$ is locally nilpotent, $\mathcal{T}(R)$ is a subgroup. We first assume that $\ch (R)=p$, a prime number. Suppose also that $a,b\in \mathcal{T}(R)$ and $a+b\neq 0$. Since $b^{-1}a \in \mathcal{T}(R)$, the extension $\mathbb F_p(a^{-1}b)$ is a finite field. Thus, there exists a positive integer $m$ such that $(a^{-1}b)^{p^m}=a^{-1}b$, and hence $$(1+a^{-1}b)^{p^m}= 1+a^{-1}b.$$ This shows that $1+a^{-1}b\in \mathcal{T}(R)$, and so $$a+b=a( 1+a^{-1}b)\in \mathcal{T}(R).$$

Now, let us write $\ch(R)=p_1^{n_1}p_2^{n_2}\dots p_k^{n_k}$, where all $p_i$'s are distinct primes. So, we have the decomposition $R\cong \prod_{i=1}^k R_i$, where, for each $i$, $R_i$ is a ring of characteristic $p_i^{n_i}$. As for every $i$, the group ${\rm U}(R_i)$ is also locally nilpotent, Lemma~\ref{lift} ensures that it suffices to show the inclusion $\mathcal{T}(R_i)+\mathcal{T}(R_i)\subseteq {\rm P}(R_i)$, for each $i$. That is why, we may assume for our convenience that $R$ itself has a prime power characteristic.

Now, $pR$ is obviously a nil-ideal of $R$, and thus it follows that ${\rm U}(R/(pR))$ is locally nilpotent. Let $a, b\in \mathcal{T}(R)$. Then, $\bar a,\bar b\in\mathcal{T}(R/(pR))$. Therefore, under presence of the first part, either $\bar a+\bar b=\bar 0$ or $\bar a+\bar b\in \mathcal{T}(R/(pR))$. In the first situation, one checks that $a+b\in {\rm Nil}(R)$. In the second situation, one inspects that $(a+b)^t=1+c$, where $t$ is a natural number and $c\in pR$ is nilpotent. Consequently, $1+c\in \mathfrak{u}(R)$, which with Lemma~\ref{center1} in hand gives that $1+c$ is torsion. This insures that $a+b\in\mathcal{T}(R)$, as promised.
\end{proof}


We now need the following technicality.

\begin{lemma}\label{torsion-free}
If $F$ is a field, $H$ is a non-trivial locally nilpotent group, and the group ring $FH$ is semi-clean, then $H$ is not torsion-free.
\end{lemma}

\begin{proof}
Suppose in a way of contradiction that $H$ is torsion-free. We know the following two major facts:

\medskip
\noindent{\bf Fact~1.} ${\rm U}(FH)=\{\lambda h\mid  0\neq\lambda\in F,~ h\in H \}$. (This follows from \cite[Corollary~8.5.5]{Sehgal} and its proof.)

\medskip
\noindent{\bf Fact~2.} $FH$ is a domain. (This follows from \cite[Theorem~8.5.14]{Sehgal} and \cite[Theorem~2]{Rhemtulla}.)

\medskip
In order to complete the proof of our statement, let $h\in H$ be a non-identity element; so, the elements $1, h, h^2$ are distinct. By Fact 2, we have ${\rm P}(FH)\setminus\{0\} =\mathcal{T}(FH)$. Since $FH$ is semi-clean,  $1+h+h^2$ can be written as a sum of two units in ${\rm U}(FH)$, which however cannot occur according to Fact 1. This proves our claim after all.
\end{proof}


The motivation for the next theorem  is a result about semi-clean group rings in \cite{Semi-clean GR} which states thus: Let $R$ be a {\it commutative} local ring and $G$ be an {\it abelian} group. Then, $RG$ is semi-clean if, and only if, $R$ is semi-clean and $G$ is locally finite.

\medskip

Concretely, we are now in a position to prove the following.

\begin{theorem}\label{local}
Let $R$ be a local ring and let $G$ be a group such that ${\rm U}(RG)$ is locally nilpotent. If $RG$ is semi-nil clean, then  $R$ is periodic and $G$ is locally finite.
\end{theorem}

\begin{proof}
Suppose that $RG$ is semi-nil clean. In order to show that $R$ is periodic, let $x\in R$ and write $x=a+b$, where $a\in {\rm P}(R)$ and $b\in {\rm Nil}(R)$. We know that some power of $a$ is idempotent. Therefore, $a$ is either a nilpotent element or a torsion unit. We will consider these two possible cases separately:

\medskip

\noindent{\bf Case~1:} $a\in {\rm Nil}(R)$. Recalling that $R$ is local, and thus ${\rm Nil}(R)\subseteq {\rm J}(R)$, we have $x\in {\rm J}(R)$. For the same reason, we write $x-1=a'+b'$, where $b'\in {\rm Nil}(R)$ and either $a'\in {\rm Nil}(R)$ or $a'\in \mathcal{T}(R)$. If $a'$ is nilpotent, then $x-1\in {\rm J}(R)$, which leads to a contradiction. Thus, $a'\in \mathcal{T}(R)$. Now, write $x=a'+(b'+1)$, where $b'+1$ is unipotent. According to Lemma~\ref{center1}, we yield $b'+1\in \mathcal{T}(R)$. Furthermore, Lemma~\ref{AC} applies to derive that $x\in {\rm P}(R)$.

\medskip

\noindent{\bf Case~2:} $a\in \mathcal{T}(R)$. In this case, we may write $x+1=a+(b+1)$ and, similarly as above, $x+1 \in {\rm P}(R)$, hence $x\in {\rm P}(R)$. Therefore, $R$ is periodic.

\medskip
In order to show that $G$ is locally finite, put $F:=R/{\rm J}(R)$ and $H:=G/\mathcal{T}(G)$. Note that $F$ is a periodic division ring, whence a field. Now $FH$, as being a homomorphic image of $RG$, is semi-nil clean. But Lemma~\ref{torsion-free} guarantees that $H$ is trivial, i.e., $G=\mathcal{T}(G)$. Consequently, $G$ is a locally nilpotent torsion group, and hence it is locally finite, as stated.
\end{proof}


We close this section by inquiring about the truth of the converse of Theorem~\ref{local}.

\begin{problem}
Let \( R \) be a periodic ring and \( G \) a locally finite group. Is \( RG \) then semi-nil clean?
\end{problem}


\medskip

\section{Unit Semi-Nil Clean Rings}\label{sec4}

Mimicking \cite{UNC}, we say that a ring $R$ is {\it unit nil clean}, or just {\it UNC} for short, if each of its invertible elements is nil clean, that is, ${\rm U}(R) \subseteq \text{Id}(R) + \text{Nil}(R)$, where ${\rm Id}(R)$ is the set of all idempotents in $R$.

In this section, we examine the properties of those rings $R$ in which the more general inclusion ${\rm U}(R) \subseteq \text{P}(R) + \text{Nil}(R)$ is fulfilled, calling them {\it unit semi-nil clean} rings. The ring of integers $\Z$ is a simple example of a unit semi-nil clean ring that is apparently {\it not} a UNC ring.

Same as in \cite{Lam-Danchev}, a ring $R$ is said to be a {\it UU} ring if ${\rm U}(R) = \mathfrak{u}(R)$ and, same as in \cite{Cui}, a ring $R$ is said to be a {\it $\pi$-UU} ring if, for each $a \in {\rm U}(R)$, there exists a positive integer $n(a)$, depending on $a$, such that $a^{n(a)} \in \mathfrak{u}(R)$.


\medskip

We continue with the following.

\begin{proposition}\label{pi-UU}
If $R$ is a unit semi-nil clean ring with $\ch(R)>0$ and ${\rm Nil}(R)$ is an additive subgroup of $R$, then ${\rm J}(R)$ is nil. Particularly, $R$ is a unit semi-nil clean NI ring if and only if ${\rm J}(R)={\rm Nil}(R)$ and  ${\rm U}(R)$ is torsion.
\end{proposition}

\begin{proof}
Let $x\in {\rm J}(R)$. Then, $1+x\in {\rm U}(R)$, so there exists   $a\in {\rm P}(R)$ such that $b:=x+1-a\in {\rm Nil}(R)$. As $\ch(R)>0$, we have $a':=1-a\in {\rm P}(R)$. Arguing as in the first paragraph of the proof of Theorem~\ref{NI}, we can replace $a$ with $a'$ to conclude that $x$ is nilpotent, i.e., ${\rm J}(R)$ is nil.   Particularly, if $R$ is an NI ring, ${\rm J}(R)={\rm Nil}(R)$.

Now, let $y$ be any element in ${\rm U}(R)$. Arguing as in the last paragraph of Theorem~\ref{NI}, there is a natural number $k$ such that  $y^k=e+c$, where $e\in {\rm Id}(R)$ and $c\in {\rm Nil}(R)$ such that $ec=ce$. Now, one finds that $e=y^k-c\in {\rm U}(R)$, and thus $e=1$. Therefore, $y^k\in \mathfrak{u}(R)$. But since $\mathfrak{u}(R)\subseteq \mathcal{T}(R)$, we detect that $y$ is torsion, as required.

The converse statement is trivial.
\end{proof}


As the ring of integers $\Z$ is both a unit semi-nil clean and a $\pi$-UU ring, the following problem arises rather naturally.

\begin{problem}
Is the statement of Proposition~\ref{pi-UU} also valid for rings of characteristic zero?
\end{problem}


\begin{remark}\label{ch>0}
Assume that $R$ is an NI ring with ${\rm J}(R)$ nil or, in other words, ${\rm J}(R) = {\rm Nil}(R)$. As the proof of Proposition~\ref{pi-UU} shows, if $R$ is unit semi-nil clean, then $R$ is manifestly $\pi$-UU (in fact, the assumption $\ch(R) > 0$ in Proposition~\ref{pi-UU} is redundant).
\end{remark}


We now arrive at the following.

\begin{corollary}\label{semisimple}
Let $R$ be a right (resp., left) perfect NI ring. Then, $R$ is unit semi-nil clean if and only if $R$ is periodic, if and only if ${\rm U}(R)$ is torsion.
\end{corollary}

\begin{proof}
If $R$ is unit semi-nil clean, in view of Remark~\ref{ch>0} alluded to above, $R$ is $\pi$-UU; hence, ${\rm U}(R/{\rm J}(R))$ is torsion. On the other hand, the classical Artin-Wedderburn theorem allows us to write that $$R/{\rm J}(R) \cong \prod_{i=1}^k {\rm M}_{n_i}(D_i)$$ for some division rings $D_i$ and positive integers $n_i$, where $1\leq i\leq k\in \mathbb{N}$. Thus, for each $i$, $D_i$ is a periodic ring, whence each ${\rm M}_{n_i}(D_i)$ is periodic employing Corollary~\ref{matrix1}. Therefore, $R/{\rm J}(R)$ is periodic, and so $R$ is too periodic with the aid of Lemma~\ref{lift}, as pursued.
\end{proof}


\begin{remark}\label{ch2}
If $R$ is a unit semi-nil clean ring and ${\rm U}(R)\cap \Z\neq \{\pm 1\}$, then it follows that $\ch(R)>0$. In fact, choose $k\in {\rm U}(R)\cap \Z$ and write $k=u+v$, where $u\in {\rm P}(R)$ and $v\in {\rm Nil}(R)$. If $R$ is of zero characteristic, by similar arguments as in the proof of Lemma~\ref{center1}, we arrive at the contradiction $|k|\leq 1$. Hence, the characteristic of $R$ is positive, indeed.
\end{remark}


We now have all the ingredients necessary to state and prove the next interesting result.

\begin{theorem}\label{algebraic2}
Suppose that $R$ is an algebraic algebra over a field. Then, $R$ is unit semi-nil clean if and only if $R$ is  periodic, if and only if ${\rm U}(R)$ is torsion.
\end{theorem}

\begin{proof}
Let $F$ be a field, and let $R$ be an algebraic $F$-algebra. Assume that $R$ is unit semi-nil clean. Now, Remark~\ref{ch2} is applicable to get that $\text{char}(R) > 0$. Consequently, thanking Lemma~\ref{center0}, we can infer that $F$ is a periodic ring.

However, $F$ is algebraic over its prime subfield $\mathbb{F}_p$, where $p=\text{char}(F)>0$, and so $R$ is algebraic over $\mathbb{F}_p$. Next, for each $\alpha \in R$, the subring $\mathbb{F}_p[\alpha]$ of $R$ generated by $\alpha$ over $\mathbb{F}_p$ is, certainly, finite. Therefore, $\alpha$ is a periodic element, proving that $R$ is periodic too.
\end{proof}


Suppose that $R$ is a unit semi-nil clean local ring. If $2\in {\rm J}(R)$, then $3\in {\rm U}(R) \cap \mathbb{Z}$; otherwise, $2\in {\rm U}(R) \cap \mathbb{Z}$. Now, Remark~\ref{ch2} gives that $\text{char}(R)>0$. Also, if $R$ is a UNC ring, then $2\in {\rm J}(R)$ in view of \cite[Lemma~2.4]{UNC}. Thus, again, $\ch(R)>0$. Hence, a similar reason as in the proof of Proposition~\ref{pi-UU} is a guarantor for the validity the following two results.

\begin{proposition}
Let $R$ be a local ring such that  ${\rm J}(R)$ is nil or ${\rm Nil}(R)$ is additively closed. Then, $R$ is a unit semi-nil clean ring if and only if ${\rm U}(R)$ is a torsion group.
\end{proposition}


\begin{proposition}
An NI ring $R$ is a UNC ring if and only if ${\rm J}(R) = {\rm Nil}(R)$ and $R$ is a UU ring.
\end{proposition}


The following technical claim is very similar to that of Lemma~\ref{lift} and so its proof is eliminated.

\begin{lemma}\label{lift2}
Suppose that $R$ and $S$ are two rings, and let $I$ be a nil-ideal of $R$. Then, the following two points hold:
\begin{itemize}
\item [(1)] If $R$ is a unit semi-nil clean ring, then $R/I$ is as well. Conversely, if $\operatorname{char}(R) > 0$ and $R/I$ is a unit semi-nil clean ring, then $R$ is so.
\item [(2)] $R$ and $S$ are unit semi-nil clean if and only if $R \times S$ is unit semi-nil clean.
\end{itemize}
\end{lemma}


We now proceed by proving the following.

\begin{proposition}
Let $R$ be a weakly $2$-primal ring.
\begin{itemize}
\item [(i)] If $\ch(R)>0$, then the polynomial ring $R[t]$ is a unit semi-nil clean ring if and only if $R$ is a unit semi-nil clean ring.
\item [(ii)] The polynomial ring $R[t]$ is a UNC ring if and only if $R$ is a UNC ring.
\end{itemize}
\end{proposition}

\begin{proof}
If $R[t]$ is unit semi-nil clean (resp., UNC), then so is $R[t]/\langle t\rangle \cong R$, as it is well known that the units of $R[t]/\langle t\rangle$ are lifted to the units of $R[t]$.

Reciprocally, suppose that $R$ is a unit semi-nil clean (resp., UNC) ring. Since the quotient-ring $R/{\rm Nil}(R)$ is a reduced ring, we have $${\rm U}\left(\frac{R}{{\rm Nil}(R)}[t]\right)={\rm U}\left(\frac{R}{{\rm Nil}(R)}\right)$$ looking at \cite[Corollary~1.7]{Leroy}. Therefore, using Lemma~\ref{lift2} (resp., \cite[Theorem~2.5]{UNC}), we can deduce that $$\frac{R[t]}{{\rm Nil}(R)[t]}\cong\frac{R}{{\rm Nil}(R)}[t]$$ is unit semi-nil clean (resp., UNC). Furthermore, since the ideal ${\rm Nil}(R)$ is locally nilpotent, ${\rm Nil}(R)[t]$ is a nil-ideal of $R[t]$. With another application of Lemma~\ref{lift2} (resp., \cite[Theorem~2.5]{UNC}), we can conclude that $R[t]$ is a unit semi-nil clean (resp., UNC) ring, as expected.
\end{proof}


We can now state without proof the following generalization of \cite[Corollary~2.12]{UNC}.

\begin{corollary}
A weakly $2$-primal ring $R$ is a UNC ring if and only if $R[t]$ is a UNC ring, if and only if ${\rm J}(R)={\rm Nil}(R)$ and ${\rm U}(R)=1+{\rm J}(R)$.
\end{corollary}


\begin{remark}
If $R$ is any ring of positive characteristic, then the ring $R[[t]]$ of formal power series is definitely \textit{not} unit semi-nil clean, as the central element $1+t\in {\rm U}(R[[t]])$ (and hence $t$) would be a periodic element, which is impossible.
\end{remark}


The next assertion is a plain consequence of our previous results, so that we can omit its proof.

\begin{corollary}
Let $R$ and $S$ be two rings of positive characteristic, $n \geq 2$, and let $M$ be an $(R,S)$-bi-module. Then, the following three statements are equivalent:
\begin{itemize}
\item [(1)] $R$ and $S$ are unit semi-nil clean rings.
\item [(2)] ${\rm T}(R, S,M)$ and ${\rm T}_n(R)$ are unit semi-nil clean rings.
\item [(3)] $R[t]/\langle t^n\rangle$ is a unit semi-nil clean ring.
\end{itemize}
\end{corollary}


Finally, for  group rings, we state the following partial results.

\begin{proposition}\label{group ring3}
Let $R$ be a ring and let $G$ be a group.
\begin{itemize}
\item [(1)] If $RG$ is unit semi-nil clean, then $R$ is also  unit semi-nil clean. Moreover, ${\rm Z}(G)$ is a locally finite group provided $\operatorname{char}(R) > 0$.
\item [(2)] If $R$ is unit semi-nil clean, $p \in \operatorname{Nil}(R)$, and $G$ is a locally finite $p$-group, where $p$ is a prime, then $RG$ is unit semi-nil clean.
\item[(3)] If $R$ is  right (resp., left) perfect, ${\rm U}(R)$ is torsion, and $G$ is locally finite, then $RG$ is a periodic ring.
\end{itemize}
\end{proposition}

\begin{proof}
Suppose that $RG$ is a unit semi-nil clean ring and let $\alpha\in {\rm U}(R)$. Then, there exist $\beta\in {\rm P}(RG)$ and $\gamma\in{\rm Nil}(RG)$ such that $\alpha=\beta+\gamma$. Assume the map $\varepsilon :RG \to R$ is the augmentation homomorphism. Furthermore, we have $\alpha= \varepsilon(\beta)+\varepsilon(\gamma)$, where $\varepsilon(\beta)\in{\rm P}(R)$ and $\varepsilon(\gamma)\in{\rm Nil}(R)$. Thus, $R$ is obviously a unit semi-nil clean ring. Moreover, ${\rm Z}(G)$ is locally finite in accordance with Lemma~\ref{center0}. This proves~(1).

Suppose now that $p \in \operatorname{Nil}(R)$ and $G$ is a locally finite $p$-group. Referring to \cite[Proposition~16]{Conn}, one verifies that $\Delta(G)$ is nil, where $\Delta(G)$ is the augmentation ideal of $RG$. Thus, (2) follows from Lemma~\ref{lift2} with the aid of the isomorphism $RG/\Delta(G) \cong R$.

The proof of (3) is similar to the proof of \cite[Theorem~1.6]{ABD} and, therefore, we leave out the arguments.
\end{proof}


As an immediate consequence, we derive the next statement.

\begin{corollary}
Let $R$ be a right (resp., left) perfect ring and let $G$ be a locally nilpotent group. Then, $RG$ is periodic if and only if ${\rm U}(RG)$ is torsion, if and only if ${\rm U}(R)$ is torsion and $G$ is locally finite.
\end{corollary}


Using a combination of Corollary~\ref{semisimple} and Proposition~\ref{group ring3}, we simply deduce the next result.

\begin{corollary}
Let $R$ be a right (resp., left) perfect NI ring and let $G$ be a locally finite group. Then, $RG$ is unit semi-nil clean if and only if $RG$ is periodic, if and only if ${\rm U}(R)$ is torsion.
\end{corollary}


\medskip

\section{Some Versions of Weakly Periodic Rings}\label{sec5}

We initiate in this section the exploration of certain variants of weakly periodic rings which are relevant to the class of {\it fine rings}, defined in \cite{CL} as $R\setminus \{0\}\subseteq {\rm U}(R)+{\rm Nil}(R)$. Precisely, we are devoted here to the study of those rings for which $R\setminus \{0\} \subseteq  \mathcal{T}(R) + {\rm Nil}(R)$.

\medskip

In this aspect, we shall say that a non-zero element $a$ in a ring $R$ is {\it torsion-fine} (or just {\it t-fine} for short) if $a=u+b$ for some $u\in \mathcal{T}(R)$ and $b\in{\rm Nil}(R)$. So, in this light, a {\it t-fine ring} is a ring in which all non-zero elements are t-fine.

Note that, if $R$ is a fine ring, then by \cite[Theorem~2.3]{CL} the ring $R$ is simple; in particular, ${\it J}(R)=\{0\}$. 


\begin{remark}
The following simple observations are useful.
\begin{itemize}
\item[(i)] Evidently, each t-fine ring is weakly periodic, but the converse is {\it not} true in general. To see this, let $R=\prod_{n=1}^\infty {\rm M}_n(\mathbb Z_2)$. Thus, in virtue of \cite[Example~3.1]{Ser}, $R$ is weakly periodic, but $R$ is not a fine ring as it is not simple.
\item[(ii)] A finite (periodic) ring need not be t-fine (for example, just consider the ring $\mathbb Z_4$).
\item[(iii)] If $R$ is both a periodic ring and a fine ring, then $R$ is t-fine.
\item[(iv)] There is a fine ring that is {\it not} t-fine. For instance, the rational field $\mathbb Q$ is obviously a fine ring, but its multiplicative group is not torsion,  so this field is {\it not} t-fine. Even more, each non-commutative division ring $D$ is fine, but, as the next lemma shows, $D$ is {\it not} t-fine.
\end{itemize}
\end{remark}


\begin{lemma}
Suppose that  $R$ is a t-fine ring. If ${\rm Nil}(R)$ is additively closed or $R$ is a local ring, then $R$ is a locally finite field.
\end{lemma}

\begin{proof}
Suppose that  ${\rm Nil}(R)$ is additively closed. Thanks to \cite[Proposition~2.5]{CL}, we know that ${\rm Nil}(R)=\{0\}$. Thus, $R\setminus \{0\} =  \mathcal{T}(R)$. That is, $R$ is a division ring. If $R$ is a local ring, then, again, $R$ is a division ring by \cite[Corollary~2.4]{CL}.
Thus, in both situations, $R$ is a locally finite field, as stated.
\end{proof}


Thus, as a consequence, a commutative ring $R$ is t-fine if, and only if, $R$ is a locally finite field. We, however, will extend this result in Corollary~\ref{comm-fine} listed below.

\medskip

We next are able to establish a non-trivial strengthening of \cite[Theorem~3.1]{CL}.

\begin{theorem}\label{matrix3}
For every t-fine ring $R$, the matrix ring ${\rm M}_n(R)$ is too t-fine for any $n\geq 1$.
\end{theorem}

\begin{proof}
Firstly, it is easy to see that, if $a$ is a t-fine element in a ring $S$ and $b\in S$ is similar to $a$ (i.e., $b=uau^{-1}$ for some $u\in{\rm U}(S)$), then $b$ is also t-fine.

If $|R|=2$, then $R\cong\mathbb F_2$. Thus, we invoke \cite[Theorem~3.1]{CL}, inferring that ${\rm M}_n(\mathbb F_2)$ is t-fine.

Assume now that $|R|>2$. The proof will be proceeded by induction on $n$. The case $n=1$ being trivial, suppose $n\geq 2$. Let $M=\begin{pmatrix} A & \beta \\ \gamma & d \end{pmatrix}\in {\rm M}_n(R)$, where $A\in {\rm M}_{n-1}(R)$. Since  $|R|>2$, \cite[Theorem~2.8]{CL} can be applied to write that $1=u+v$, where $u,v\in {\rm U}(R)$. Therefore, using \cite[Proposition~3.9]{CL} and combining it with what we saw in the first paragraph about the invariance of t-fine elements by similarity, we may assume that both $A$ and $d$ are non-zero elements.

Write, by induction hypotheses, that $A=U+N$, where $U$ is a torsion unit and $N$ is a nilpotent element in ${\rm M}_{n-1}(R)$. Similarly, write $d=v+t$, where $v\in{\rm U}(R)$ and $t\in{\rm Nil}(R)$. Then, one represents that 
$$M=\begin{pmatrix} U & 0 \\ \gamma & v \end{pmatrix} + \begin{pmatrix} N & \beta \\ 0 & t \end{pmatrix}.$$
It is easily seen that the second matrix on the right-hand side above is nilpotent, and the first one, $X$ say, is $\pi$-UU (indeed, since $R$ is t-fine, one may write $v^m=1$ for some $m\geq 1$; in addition, write $U^k=1$ for some $k\geq 1$. Hence, $X^{km}$ possesses diagonal entries equal to $1$ whence it is unipotent), and so torsion with the help of Lemma~\ref{center1}. Thus, $M$ is t-fine, as required.
\end{proof}


The following  significant corollary is a non-trivial consequence of Theorem~\ref{matrix3}.

\begin{corollary}\label{comm-fine}
Let $R$ be a commutative ring and let $n$ be a positive integer. Then, ${\rm M}_n(R)$ is a t-fine ring if and only if $R$ is a locally finite field.
\end{corollary}

\begin{proof}
Assume that ${\rm M}_n(R)$ is a t-fine ring and let $a$ be a non-zero element in $R$. Thus, one detects that $aE_{11}$ is a t-fine element in ${\rm M}_n(R)$, where $E_{11}$ is the matrix unit whose $(1,1)$ entry is $1$ and all other entries are $0$. So, \cite[Lemma~4.7]{CL} and the fact that $R$ is commutative are in use to deduce that $aR=R=Ra$, i.e., $a\in{\rm U}(R)$. Consequently, $R$ is a locally finite field. This gives the ``only if'' part of the statement.

The ``if'' part follows automatically from Theorem~\ref{matrix3}, concluding the argumentation.
\end{proof}


Let ${\rm End}(G)$ be the endomorphism ring of an abelian group $G$. It was shown in \cite[Theorem 3.5]{ABD} that ${\rm End}(G)$ is periodic exactly when $G$ is a finite group. In this direction, we obtain the following.

\begin{corollary}
Let $G$ be an abelian group. Then, ${\rm End}(G)$ is a t-fine ring if and only if $G$ is a finite elementary abelian $p$-group for some prime $p$.
\end{corollary}

\begin{proof}
Set $R:={\rm End}(G)$. If $R$ is a t-fine ring, then it is simple. Thus, consulting with \cite[Theorem~111.2]{Fuchs}, $G$ is either a finite direct sum of copies the rational group $\mathbb{Q}$, or is a finite direct sum of the cyclic $p$-group $\mathbb{Z}_p$ for some prime number $p$. But, in the former case, it follows from \cite{Fuchs} that $R \cong \mathrm{M}_n(\mathbb{Q})$ for some $n \geq 1$, which is known by what we have established in Corollary~\ref{comm-fine} to be not t-fine. Therefore, the second case occurs, i.e., $G$ is a finite elementary abelian $p$-group, as asserted.

Oppositely, if $G$ is a direct sum of $s$ copies of $\mathbb{Z}_p$ for some positive integer $s$ and some prime number $p$, then one knows again from \cite{Fuchs} that $R \cong \mathrm{M}_s(\mathbb{Z}_p)$, which ring is t-fine appealing to Corollary~\ref{comm-fine}, as claimed.
\end{proof}

It is worthwhile noticing that the last corollary means that if ${\rm End}(G)$ is a t-fine ring for some abelian group $G$, then ${\rm End}(G)$ is a periodic ring, but the reverse claim manifestly fails.


\medskip

We finish the present work by posing the following two questions, which are closely related to Problem~\ref{1} quoted above and which seem very difficult to be answered at this stage.

\begin{problem}
Is there a t-fine ring that is \textit{not} periodic? Also, is any t-fine ring clean or even strongly $\pi$-regular?
\end{problem}

Notice that the first query will surely be settled in the affirmative if we succeed to construct a t-fine ring whose unit group is {\it not} torsion.


\bigskip

\noindent{\bf Acknowledgments.} The research work of M. H. Bien was funded by Vietnam National Foundation for Science and Technology Development (NAFOSTED) under Grant No. 101.04-2023.18. The research work of P. V. Danchev is supported in part by the Junta de Andalucia under Grant FQM 264. The research work of M. Ramezan-Nassab is supported in part by a grant from IPM (Grant No. 1403160021).


\bigskip

\noindent{\bf Declarations.} Our statements here are the following ones:

\medskip

\begin{itemize}
\item {\bf Ethical Declarations and Approval:} The authors have no competing interests to declare that are relevant to the content of this article.

\medskip

\item {\bf Competing Interests:} The authors declare no any conflict of interest.

\medskip

\item {\bf Availability of Data and Materials:} Data sharing is not applicable to this article as no data-sets or any other materials were generated or analyzed during the current study.
\end{itemize}


\end{document}